\documentclass[10pt]{article}

\usepackage{amsmath,amsfonts,amssymb}
\usepackage{epsf}
\RequirePackage{amsfonts,amssymb,amsmath,amscd,amsthm}
\usepackage{epsfig,tabularx,graphicx,float}
\usepackage{graphicx}
\usepackage[a4paper]{geometry}

\usepackage{parskip}
\usepackage[utf8]{inputenc}
\usepackage{times}
\usepackage{float}
\usepackage{multicol}
\usepackage{multirow}

\usepackage{float}
\newtheorem{thm}{Theorem}[section]

\newtheorem{defi}[thm]{Definition}

\newtheorem{ex}[thm]{Example}

\newtheorem{lem}[thm]{Lemma}

\begin{document}

\begin{center}
\textbf{\large Wavelet Sets and Generalized Scaling Sets on Vilenkin Group}
\end{center}

\begin{center}
Prasadini Mahapatra\\
Department of Mathematics, National Institute of Technology Rourkela, Odisha, India.\\
prasadinirkl.12@gmail.com\\
\vspace{0.5cm}
Arpit Chandan Swain \\
Department of Biology, Utrecht University, The Netherlands.\\
swainarpit@gmail.com\\
\vspace{0.5cm}
Divya Singh \\
Department of Mathematics, National Institute of Technology Rourkela, Odisha, India.\\
divya\_allduniv@yahoo.com
\end{center}

\date{Received:     ; Accepted:          }

\noindent \textbf{Abstract--} For Vilenkin group only the existence of multiwavelets associated with multiresolution analysis (MRA) is known. In this paper, we have shown that by using wavelet sets we can also construct single wavelet in case of Vilenkin group which are not associated with MRA. We have given characterization of single and multi-wavelet sets on Vilenkin group. Further, we have studied generalized scaling sets, some of their properties and relation between wavelet sets and generalized scaling sets.




%

\section{Introduction}

In the last two decades several generalizations and extensions of wavelets have been introduced. Some fundamental ideas from wavelet theory, such as multiresolution analysis (MRA), have appeared in very different contexts. This paper is related to one such generalization of wavelets.

Walsh functions were introduced by J. Walsh in 1923 as linear combination of Haar functions. N. J. Fine and N. Ya Vilenkin independently determined that Walsh system is the group of characters of the Cantor dyadic group. Vilenkin introduced a large class of locally compact abelian groups, called Vilenkin groups, which includes Cantor dyadic group as a particular case. The generalised Walsh functions form an orthonormal system in the Vilenkin group $G$ and mask of refinable equation is given in terms of these generalised Walsh functions. Refinable equation gives refinable function which generates MRA and hence wavelets, if the mask satisfies certain conditions. Necessary and sufficient conditions were given over the mask of scaling function $\phi$ in terms of modified Cohen's condition and blocked sets such that $\phi$ generates an MRA.  

A lot of work related to wavelets on Vilenkin group has been done, but all the wavelets construction is based on only MRA method. In case of $L^2(\mathbb{R}^n)$ non-MRA wavelets and wavelet sets have been extensively studied by several authors \cite{dai,dai1,hern1,hern2,Ionascu,merrill}. In case of Vilenkin group if the associated prime $p$ is greater than 2, then MRA generates a multiwavelet set having $p-1$ functions. In this paper we have studied single wavelets on Vilenkin group generated by wavelet sets. The concept of generalized scaling set was given in \cite{bownik}. Generalized scaling sets determine wavelet sets and hence wavelets. Any result related to generalized scaling sets is not available in case of locally compact abelian groups. This paper consists of characterization of generalized scaling sets and their properties.

The algebraic and topological structure of Vilenkin groups is given in section 2 along with basic results on wavelets. Characterization of single/multi-wavelet sets is given in section 3 with relevant examples and section 4 consists of properties of generalized scaling sets on Vilenkin group. 


\section{Preliminaries}

\subsection{The Vilenkin Group}

For a prime $p$, Vilenkin group $ G $ is defined as the group of sequences
$$
x=(x_{j})=(..., 0, 0, x_{k}, x_{k+1}, x_{k+2},...),
$$
where $x_{j}\in \lbrace 0, 1, ..., p-1\rbrace $, for $ j \in \mathbb{Z} $ and $ x_{j}=0 $, for $ j < k = k(x)$. The group operation on $ G $, denoted by $ \oplus $, is defined as coordinatewise addition modulo $ p $:
$$
(z_{j})=(x_{j})\oplus (y_{j})\Leftrightarrow z_{j}=x_{j} + y_{j}(\text{ mod }\: p), \text{ for } j \in \mathbb{Z}.
$$
$\theta$ denotes the identity element (zero) of $G$.
Let
$$
U_{l}=\{(x_{j}) \in G : x_{j}=0\: \text{ for }\: j\leq l\} ,\quad l\in \mathbb{Z},
$$
be a system of neighbourhoods of zero in $G$. In case of topological groups if we know neighbourhood system $\{ U_{l} \}_{l \in \Bbb{Z}}$ of the identity element, then we can determine neighbourhood system of every point $x=(x_{j}) \in G$ given by $\{ U_{l} \oplus x \}_{l \in \Bbb{Z}}$, which in turn generates a topology on $G$.

Let $ U=U_{0} $ and $ \ominus $ denotes the inverse operation of $ \oplus $. The Lebesgue spaces $ L^{q}(G),\: 1\leq q\leq \infty $, are defined with respect to the Haar measure $ \mu $ on Borel subsets of $ G $ normalized by $ \mu(U)=1 $.

The group dual to $ G $ is denoted by $ G^{*} $ and consists of all sequences of the form
$$
\omega=(\omega_{j})=(...,0,0,\omega_{k},\omega_{k+1},\omega_{k+2},...),
$$
where $ \omega_{j} \in \lbrace 0,1,...,p-1\rbrace $, for $ j \in \mathbb{Z} $ and $ \omega_j=0 $, for $ j<k=k(\omega) $. The operations of addition and subtraction, the neighbourhoods $\lbrace U_{l}^{*}\rbrace $ and the Haar measure $ \mu^{*} $ for $ G^{*} $ are defined similarly as for $G$. Each character on $ G $ is defined as
$$
\chi(x,\omega)=\exp\bigg(\frac{2\pi i}{p}\sum_{j\in \mathbb {Z}}{x_{j}w_{1-j}}\bigg),\; x \in G,
$$
for some $ \omega \in G^{*} $.

Let $H = \lbrace (x_{j}) \in G\: |\: x_{j} =0, \; \text{ for } j>0\rbrace$ be a discrete subgroup in $G$  and $A$ be an automorphism on $G$ defined by $(Ax)_j=x_{j+1}$, for $x=(x_j) \in G$. From the definition of annihilator and above definition of character $\chi$, it follows that the annihilator $H^{\perp}$ of the subgroup $H$ consists of all sequences $(\omega_j) \in G^{*}$ which satisfy $\omega_j=0$, for $j>0$.

Let $ \lambda:G\longrightarrow \mathbb R_+$ be defined by
\begin{equation*}
\lambda(x)=\sum_{j \in \mathbb{Z}}{x_{j} p^{-j}}, \qquad x=(x_j) \in G.
\end{equation*}
It is obvious that the image of $H$ under $\lambda$ is the set of non-negative integers $\mathbb Z_{+}$. For every $\alpha \in \mathbb Z_{+}$, let $h_{[\alpha]}$ denote the element of $H$ such that $\lambda(h_{[\alpha]})=\alpha$. For $G^{*}$, the map $ \lambda^{*}:G^{*}\longrightarrow \mathbb R_{+}$, the automorphism $B \in \text{Aut } G^{*}$, the subgroup $U^{*}$ and the elements $\omega_{[\alpha]}$ of $H^{\perp}$ are defined similar to $\lambda$, $A$, $U$ and $h_{[\alpha]}$, respectively. 

The generalized Walsh functions for $ G $ are defined by
\begin{equation*}
W_\alpha(x)=\chi(x,\omega_{[\alpha]}),\quad \alpha \in \mathbb Z_+, x \in G.
\end{equation*}
These functions form an orthogonal set for $L^2(U)$, that is,
\begin{equation*}
\int_{U}{W_{\alpha}(x) \overline{W_{\beta}(x)}d\mu(x)}=\delta_{\alpha,\beta},\quad \alpha,\beta \in \mathbb Z_{+},
\end{equation*}
where $ \delta_{\alpha,\beta} $ is the Kronecker delta. The system $ {W_{\alpha}} $ is complete in $ L^{2}(U) $. The corresponding system for $ G^{*} $ is defined by
\begin{equation*}
W_{\alpha}^{*}(\omega)=\chi(h_{[\alpha]},\omega),\qquad \alpha \in \mathbb Z_{+}, \omega \in G^{*}.
\end{equation*}
The system $\lbrace W_{\alpha}^{*} \rbrace$ is an orthonormal basis of $ L^{2}(U^{*}) $.

For positive integers $ n $ and $ \alpha $,
\begin{equation*}
U_{n,\alpha}=A^{-n}(h_{[\alpha]}) \oplus A^{-n}(U).
\end{equation*}

\subsection{Wavelets on Vilenkin group}

In \cite{lang1} Lang constructed compactly supported orthogonal wavelets on the locally compact Cantor dyadic group. These wavelets were identified with Walsh series on the real line and included the Haar basis. By invoking the Calederon-Zygmund integral operator theory Lang \cite{lang2} proved that if a wavelet on Cantor dyadic group satisfies Lipschitz type regularity condition then the wavelet series converges unconditionally in $L^q$ (for $q > 1$). Farkov \cite{farkov,farkov1,farkov2} extended the results of Lang to Vilenkin groups. In \cite{farkov}, Strang-fix condition, partition of unit property and the stability of scaling functions were considered. Further, necessary and sufficient conditions were given over the mask of scaling function $\phi$ in terms of modified Cohen's condition and blocked sets such that $\phi$ generates an MRA. Farkov also gave an algorithm to construct orthogonal wavelets in $L^2(G)$. In case of Vilenkin group if the associated prime $p$ is greater than 2, then MRA generates a set having $p-1$ functions.


\begin{defi}\cite{farkov} Let $ L_{c}^{2}(G) $ be the set of all compactly supported functions in $ L^{2}(G) $. A function $ \phi \in L_{c}^{2}(G) $ is said to be a \textit{refinable function}, if it satisfies an equation of the type
\begin{equation}
\phi(x)=p\sum_{\alpha =0}^{p^{n}-1}{a_{\alpha}\phi (Ax\ominus h_{[\alpha ]})}.
\end{equation}
The above functional equation is called the \textit{refinement equation}. The generalized Walsh polynomial
\begin{equation}
m(\omega)=\sum_{\alpha =0}^{p^{n}-1}{a_{\alpha} \overline{W_{\alpha}^{*}{(\omega)}}}
\end{equation}
is called the \textit{mask} of the refinement equation (or the mask of its solution $ \phi $).
\end{defi}

\begin{thm}\cite{farkov}
\textit{Let $ \phi \in L_{c}^{2}{(G)} $ be a solution of the refinement equation, and let $ \widehat{\phi}(\theta)=1 $. Then,
$$
\sum_{\alpha =0}^{p^{n}-1}{a_{\alpha}}=1, \qquad \text{supp } \phi \subset U_{1-n},
$$
and
$$ 
\widehat{\phi}(\omega)=\prod_{j=1}^{\infty}{m(B^{-j}\omega)}.
$$
Moreover, the following properties are true:\\
$1$. $ \widehat{\phi}(h^{*})=0 $, for all $ h^{*} \in H^{\perp}\setminus {\lbrace \theta\rbrace} $ (the modified Strang-Fix condition),\\
$2$. $ \sum_{h \in H}{\phi (x\oplus h)}=1 $, for almost every $ x \in G $ (the partition of unit property)}.
\end{thm}

\begin{defi}\cite{farkov}
A set $ M\subset U^* $ is said to be \textit{blocked} (for the mask $ m $) if it coincides with some union of the sets $ U_{n-1,s}^{*}$, $0\leq s\leq p^{n-1}-1 $, does not contain the set $ U_{n-1,0}^{*} $, and satisfies the condition
$$
T_{p}{M} \subset M \cup \lbrace \omega \in U^{*} : m(\omega)=0\rbrace,
$$
where
$$
T_pM=\bigcup_{l=0}^{p-1} \lbrace B^{-l}\omega_{[l]}+B^{-1}(\omega): \omega \in M \rbrace.
$$
\end{defi}

\begin{defi}
A collection $(V_j)_{j \in \mathbb Z}$ of closed subspaces of $L^{2}{(G)}$ is called a Multiresolution analysis (MRA) if the following conditions are satisfied:\\
(i) $ V_{j} \subset V_{j+1} $, for all $ j \in \mathbb Z $ \\
(ii) $ \overline{\cup_{j \in \mathbb Z} V_{j}}=L^{2}{(G)} $ and $ \cap_{j \in \mathbb Z} V_{j}=\lbrace 0\rbrace $ \\
(iii) $ f(\cdot) \in V_{j}\Leftrightarrow f(A\cdot) \in V_{j+1} $, for all $ j \in \mathbb Z $\\
(iv) $ f(\cdot) \in V_{0}\Rightarrow f(\cdot \ominus h) \in V_{0} $, for all $ h \in H $\\
(v) there is a function $ \phi \in L^{2}{(G)} $ such that the system $ \lbrace \phi(\cdot \ominus h)\vert h \in H\rbrace $ is an orthonormal basis of $ V_{0} $.
		
The function $ \phi $ in condition (v) is called scaling function of the MRA $(V_j)_{j \in \mathbb Z}$.
\end{defi}

For $ \phi \in L^{2}{(G)} $,
$$
\phi_{j,h}(x)=p^{j/2}{\phi (A^{j}{x}\ominus h)}, \qquad j \in \mathbb Z, \;\; h \in H
$$
and the system $\lbrace \phi_{j,h} : h \in H\rbrace$ forms an orthonormal basis of $ V_{j} $, for every $ j \in \mathbb Z $.

\begin{thm}\cite{behera}
A function $\phi \in L^{2}(G)$ is a scaling function for an MRA of $L^{2}(G)$ if and only if
\begin{enumerate}
\item $\sum_{h \in H^{\perp}}|\hat{\phi}(\omega \oplus h)|^2=1$, for a.e $\omega \in G^{*}$
\item $lim_{j \rightarrow \infty}|\hat{\phi}(B^{-j}\omega)|=1$, for a.e $\omega \in G^{*}$
\item $\hat{\phi}(B\omega)=m(\omega)\hat{\phi}(\omega)$, for a.e $\omega \in G^{*}$.
\end{enumerate}
\end{thm}

A function $ \phi $ is said to generate an MRA in $ L^{2}{(G)} $ if the system $ \lbrace \phi(\cdot \ominus h)\vert h \in H\rbrace $ is orthonormal in $ L^{2}{(G)} $ and, the family of subspaces
$$
V_{j}=\overline{\text{span}}\lbrace \phi_{j,h} : h \in H\rbrace, \qquad j \in \mathbb Z,
$$
forms an MRA in $ L^{2}{(G)}$ with scaling function $\phi$. Farkov gave the following condition under which a compactly supported function $\phi \in L^2(G)$ generates an MRA.

\begin{thm}\cite{farkov}
Suppose that the refinement equation possesses  a solution $\phi$ such that $\widehat{\phi}(\theta)=1$ and the corresponding mask $m$ satisfies the conditions
$$
m(\theta)=1 \text{ and } \Sigma_{l=0}^{p-1} \vert m(\omega \oplus \delta_{l}) \vert^2=1, \text{ for } \omega \in G^*,
$$
where $\delta_l$ is the sequence $\omega=(\omega_j)$ such that $\omega_1=l$ and $\omega_j=0$ for $j \neq 1$. Then the following are equivalent:\\
(a) $\phi$ generates an MRA in $L^2(G)$. \\
(b) $m$ satisfies the modified Cohen's condition, i.e. there exists a compact subset $E$ of $G^*$ containing a neighborhood of zero such that $E$ is translation congruent to $U^*$ modulo $H^{\perp}$ and $\inf_{j \in \mathbb N} \inf _{ \omega \in E} |m(B^{-j}\omega)| > 0$.\\
(c) $m$ has no blocked sets.
\end{thm}

Using the above characterization of refinable function and the matrix extension method Farkov gave an algorithm for the construction of orthonormal wavelets $ \psi_{1},...,\psi_{p-1} $ such that the functions
$$
\psi_{l,j,h}(x)= p^{j/2}{\psi_{l}{(A^{j}x\ominus h)}}, \qquad 1\leq l\leq p-1, j \in \mathbb Z, \; \; h \in H,
$$
form an orthonormal basis of $ L^{2}{(G)} $.

\section{Wavelet sets}

\begin{defi}
A measurable subset $\Omega$ of $G^*$ is called a \textit{wavelet set} if $\psi=\check{1}_{\Omega}$ is a wavelet in $L^{2}(G)$, where $1_{\Omega}$ is the characteristic function on $\Omega$.
\end{defi}

Let $E$ and $F$ be two measurable subsets of $G$. Then $E$ is said to be $H$-translation congruent to $F$, if there exists a partition $\lbrace E_n : n \in N \subseteq  \mathbb Z^+\rbrace$ of $E$ such that $\lbrace E_n\oplus h_{[n]} : h_{[n]} \in H'\rbrace$, where $\lambda(H')=N$, is a partition of $F$. 

\begin{thm}
Suppose that $S$ is a measurable subset of $G^{*}$ such that $\bigcup_{h \in H^{\perp}}(S\oplus h) = G^{*}$ a.e. Then the following statements are equivalent.
\begin{enumerate}
\item $S \cap (S \oplus h) = \phi$ a.e., whenever $h$ is a non-zero element in $H^{\perp}$.
\item $\mu^*(S)=1$.
\end{enumerate} 
\end{thm}
	
\begin{proof}
Let $f(\omega)=\sum_{h \in H^{\perp}}1_{S}(\omega \ominus h)$. Then if $S$ satisfies property (1), it follows that $f\equiv 1$ and, we may write
\begin{align*}
\mu^*(S)&=\int_{G^{*}}1_{S}(\omega)d\omega\\
		&=1.
		\end{align*}
For converse suppose that (2) holds, assumption on $S$ implies that $f(\omega)\geq 1$. Also observe that, 
\begin{align*}
\int_{U^{*}}f(\omega)d\omega &=\int_{U^{*}}\sum_{h \in H^{\perp}}1_{S}(\omega \ominus h)d\omega\\
		                     &=\mu^*(S)=1
\end{align*}
This implies $f \equiv 1$ and thus $S \cap (S \oplus h) = \phi$ a.e.
\end{proof}

In \cite{bened1,bened2} the authors considered a locally compact abelian group (LCAG) $G$, having a compact open subgroup $H$. The difference between the wavelet theory developed on $\mathbb R^n$ and other similar type of structures, and the one given on LCAG in \cite{bened1} is that in the former elements of a non-trivial discrete subgroup are used for translations, while a LCAG  may not possess such type of subgroup. To overcome this difficulty in \cite{bened1} an operator $\tau_{[s]}$ was constructed for each element $[s]$ of the discrete quotient group $G/H$. These operators were determined by a choice $\mathcal D$ of coset representatives in the dual group $\hat{G}$ of $G$, for $\hat{G}/H^{\perp}$. Elements of a  countable non-empty subset $\mathcal A$ of Aut$(G)$, the group under composition of homeomorphic automorphisms of $G$, were used for dilations. Then similar to the wavelet sets on Euclidean spaces they defined translation and dilation congruences and gave characterizeation of multiwavelet sets.

On the same lines we have given characterization for single wavelet sets on Vilenkin group $G$. Similar result holds for multiwavelet sets. However, in case of Vilenkin group the elements of $H$ in $G$, and $H^{\perp}$ in $G^*$ are used for translation and instead of using an arbitrary countable set of automorphisms on $G$, integral powers of the automorphisms $A$ and $B$ are used for dilations in $G$ and $G^*$, respectively.

\begin{thm}
Let $ G $ be the Vilenkin group and $ H $ be the discrete subgroup of $ G $. Let $ U^{*} $ be the neighborhood of $ \theta $ in $ G^{*} $. Let $ \Omega $ be a measurable subset of $ G^{*} $. Then, $ \Omega $ is a wavelet set if and only if both of the following conditions hold:\\
(a) $\lbrace B^{n}\Omega : n \in \mathbb Z\rbrace$ tiles $G^{*}$ upto sets of measure zero, and\\
(b) $\Omega $ is $H^{\perp}$-translation congruent to $U^{*}$ upto sets of measure zero.
\end{thm}

\begin{thm}
Let $ G $ be the Vilenkin group and $ H $ be the discrete subgroup of $ G $. Let $ U^{*} $ be the neighborhood of $ \theta $ in $ G^{*} $. Let $\lbrace \Omega_i \rbrace $, $1 \leq i \leq p-1$, be measurable subsets of $ G^{*} $. Then, $ \lbrace \Omega_1, \Omega_2, \cdots, \Omega_{p-1} \rbrace$ forms a multiwavelet set if and only if both of the following conditions hold:\\
(a) $ \lbrace B^{n}\Omega_i : n \in \mathbb Z, 1 \leq i \leq p-1\rbrace $ tiles $ G^{*} $ upto sets of measure zero, and\\
(b) Each $ \Omega_i $, $1 \leq i \leq p-1$, is $H^{\perp}$-translation congruent to $ U^{*} $ upto sets of measure zero.
\end{thm}

\section{Generalized scaling sets}

Throughout this section we will denote the elements of $G^*$ as $(\cdots \omega_{-2} \omega_{-1} \omega_0.\omega_1 \omega_2 \cdots)$. In the beginning some results related to the translation map $\rho$ are given as in \cite{papadakis}. Thereafter generalized scaling sets are introduced on the Vilenkin group.

Let us define a map $\rho:G^{*}\rightarrow U^{*}$ such that $\omega=\rho(\omega)\oplus l(\omega)$, for some $l(\omega)\in H^{\perp}$. It can be seen easily that $\rho$ is not one-one, but it is onto. Here $\rho_{|U^{*}}$ is the identity map. Let  $M$ be a measurable subset of $G^{*}$, then $M$ is $H^{\perp}-$ translation congruent to $U^{*}$ iff $\rho_{|M}$ is a bijection between $M$ and $U^{*}$.  
	
Let us consider a map $I:U^{*}\rightarrow U^{*}$ defined by
$$
I(\omega)=\left\{\begin{matrix}
B\omega, & \omega \in U^{*}_{1}\\
B(\omega \oplus 0.p-\sigma), & \omega \in U^{*}\setminus U^{*}_{1},
\end{matrix}\right.
$$
where $\sigma \in \{1,2,...,p-1\}$ is the 1-th coordinate of $\omega \in U^{*}\setminus U^{*}_{1}$.
Clearly, the map $I$ is onto but not one-one. 

Note that, $I^{j}:U^{*} \rightarrow U^{*}$, for all $j \geq 0$, where $I^{0}$ is the identity map. Furthermore, it follows from the definition of $I$ that
$$
I^{j}(\omega)=B^{j}\omega \oplus l(\omega), \text{ for some }l(\omega) \in H^{\perp}, \;\; \text{and } j=0,1,2,\cdots
$$
Since $I^{j}(\omega) \in U^{*} $ and $\rho$ is $H^{\perp}-$ periodic, therefore
$$
I^{j}(\omega)=\rho(I^{j}(\omega))=\rho(B^{j}\omega \oplus l(\omega))=\rho(B^{j}\omega).
$$

\begin{lem}
For	$\omega_{1},\omega_{2} \in U^{*}$, $I(\omega_{1})=I(\omega_{2})$ iff one of the following conditions is true.\\
i. $\omega_{1} \in U^{*}\setminus U^{*}_{1}$, $\omega_{2} \in U^{*}_{1}$ and $\omega_{1} \ominus \omega_{2}=0.\sigma$, or $\omega_{1}=\rho(\omega_{2}\oplus 0.\sigma)$, where $\sigma=(\omega_{1})_{1}$.\\
ii. $\omega_{1},\omega_{2} \in U^{*}_{1}$ and $\omega_{1}=\omega_{2}$.\\
iii. $\omega_{1},\omega_{2} \in U^{*}\setminus U^{*}_{1}$ and they differ only at the 1-th coordinate.
\end{lem}
	
\begin{proof}
Suppose that $\omega_{1},\omega_{2} \in U^{*}$ with  $I(\omega_{1})=I(\omega_{2})$. If $\omega_{1} \in U^{*}\backslash U^{*}_{1}$ and $\omega_{2} \in U^{*}_{1}$,  then  $I(\omega_{1})=B\omega_{1} \oplus p-\sigma.0$ and $I(\omega_{2})=B\omega_{1}$. Since $I(\omega_{1})=I(\omega_{2})$, $B\omega_{1} \oplus p-\sigma.0=B\omega_{2}$ implies that $\omega_{1} \ominus \omega_{2}=0.\sigma$. Thus condition (i) holds. Statements (ii) and (iii) follow trivially. Further, it is easy to see that if any one of the three conditions holds, then $I(\omega_{1})=I(\omega_{2})$.
\end{proof}

\begin{thm}
Suppose that $\mathcal{U}$ is a measurable subset of $U^{*}$ satisfying the condition $\mathcal{U}^{c}=U^{*}\setminus \mathcal{U}=\bigcup_{\sigma=1}^{p-1}\rho(\mathcal{U} \oplus 0.\sigma)$. Let $\Upsilon_{0}=\mathcal{U}$ and $\Upsilon_{n}=\mathcal{U}\cap I^{-1}(\Upsilon_{n-1})$, for $n \geq 1$. Then, the set $B^{(n+1)}\Upsilon_{n}$ is $H^{\perp}-$ translation congruent to $U^{*}$, for $n \geq 0$.
\end{thm}

\begin{proof}
Here we have to prove that $B^{(n+1)}\Upsilon_{n}$ is $H^{\perp}-$ translation congruent to $U^{*}$, for which it is enough to show that $\rho_{| B^{(n+1)}\Upsilon_{n}}$ is bijection between $B^{(n+1)}\Upsilon_{n}$ and $U^{*}$. Note that $I$ maps $U^{*}$ onto $U^{*}$ and each $\omega^{'} \in U^{*}$ is the image of $p$ points, $\omega$ and $\rho(\omega \oplus 0.\sigma)$, $\sigma \in \lbrace 1, 2, \cdots, p-1\rbrace$ in $U^{*}$ under the map $I$. By our assumption, if $\omega \in \mathcal{U}$, then $\bigcup_{\sigma=1}^{p-1}\rho(\omega \oplus 0.\sigma) \in U^{*} \setminus \mathcal{U} $. Thus, $\gamma=I_{|\mathcal{U}}:\mathcal{U}\rightarrow U^{*}$ is one-one and onto.

Clearly, $\gamma^{-1}\mathcal{U} \subset \mathcal{U}$ and thus,
$$
\Upsilon_{1}=\mathcal{U}\cap I^{-1}(\Upsilon_{0})\\
=\gamma^{-1}(\mathcal{U})
$$

This map $\gamma_{|\Upsilon_{1}}:\Upsilon_{1}\rightarrow \mathcal{U}$ is one-one and onto because $\gamma$ is one-one. Let $\gamma^{2}=\gamma o\gamma_{|\Upsilon_{1}}:\Upsilon_{1}\rightarrow U^{*}$. Thus $\gamma^{2}$ is one-one, onto and $\Upsilon_{2}=\gamma^{-1}(\Upsilon_{1})$.

Again suppose that,
$$
\Upsilon_{k}=\mathcal{U} \cap I^{-1}(\Upsilon_{k-1})=\gamma^{-k}(\mathcal{U})
$$
then,
$$
\Upsilon_{k+1}=
\gamma^{-1}(\Upsilon_{k}).
$$
Thus by induction, we have $\Upsilon_{n}=\gamma^{-1}(\Upsilon_{n-1})=\gamma^{-n}(\mathcal{U})$ and
\begin{eqnarray}
\gamma o\gamma_{|\Upsilon_{n}}^{n}=\gamma^{n+1}:\Upsilon_{n}\rightarrow U^{*}
\end{eqnarray}
is one-one and onto.
		
Since (5) is true for each $n \geq 0$, fix $n \in \mathbb{N} \cup \{0\}$ and let $\beta:B^{(n+1)}\Upsilon_{n}\rightarrow \Upsilon_{n}$ be the one-one and onto map defined by $\beta(\omega)=B^{-(n+1)}\omega$.

Therefore,
\begin{eqnarray*}
\rho_{|B^{(n+1)}\Upsilon{n}}=\gamma^{n+1}o\beta
\end{eqnarray*}
		
Hence $\rho_{| B^{(n+1)}\Upsilon_{n}}:B^{(n+1)}\Upsilon_{n}\rightarrow U^{*}$ is one-one, onto and thus, $B^{(n+1)}\Upsilon_{n}$ is $H^{\perp}-$ translation congruent to $U^{*}$.
\end{proof}

\begin{defi}
A measurable subset $S$ of $G^{*}$ is called a \textit{generalized scaling set} if $\mu^*(S)=\dfrac{1}{p-1}$ and $(BS \setminus S)$ is a wavelet set.
\end{defi}

\begin{ex}\textbf{Generalized scaling set}\\
Let $\Omega$ be given by Example 3.5, then it follows easily that for $S=\bigcup_{j=1}^{\infty}B^{-j}\Omega$, $\mu^*(S)=\dfrac{1}{p-1}$ and $BS \setminus S$ is a wavelet set. That is, $S$ forms a generalized scaling set.
\end{ex}

In fact, we have the following characterization of generalized scaling set in terms of wavelet set.

\begin{lem}
A measurable subset $S$ of $G^{*}$ is a generalized scaling set if and only if $S=\bigcup_{j=1}^{\infty}B^{-j}\Omega$ modulo null sets, for some wavelet set $\Omega$.
\end{lem}

\begin{proof}
Suppose that $S=\bigcup_{j=1}^{\infty}B^{-j}\Omega$ modulo null sets, for some wavelet set $\Omega$. Thus $BS \setminus S = (\bigcup_{j=1}^{\infty}B^{1-j}\Omega) \setminus (\bigcup_{j=1}^{\infty}B^{-j}\Omega)=\Omega$, as the union is disjoint. We know that $\mu^*(\Omega)=1$ and hence
\begin{equation*}
\mu^*(S)
=\sum_{j=1}^{\infty}p^{-j}\mu^*(\Omega)
=\frac{1}{p-1}.
\end{equation*}
This proves that $S$ is a generalized scaling set.

Conversely, suppose that $S$ is a generalized scaling set. Then $\Omega=BS \setminus S$ is a wavelet set.
From theorem 3.3, it follows that $\{B^{j}\Omega:j \in \mathbb{Z}\}$, or equivalently $\{S, B^{j}\Omega:j \geq 0 ,j \in \mathbb{Z}\}$ forms a partition of $G^{*}$ a.e. That implies $S \subset 
\bigcup_{j=1}^{\infty}(B^{-j}\Omega)$ a.e. Since $\bigcup_{j=1}^{\infty}B^{-j}\Omega$ and $S$ both have same measure therefore, $S=\bigcup_{j=1}^{\infty}B^{-j}\Omega$ upto sets of measure 0.
\end{proof}

\noindent \textbf{Remark.} If $S$ is a generalized scaling set and $\Omega$ is the associated wavelet set, then the following consistency equation is satisfied.
\begin{align*}
\sum_{h \in H^{\perp}} 1_{S}(B^{-1}(\omega \oplus h))
	&=\sum_{h \in H^{\perp}} 1_{\cup_{k \leq 0}B^{k}\Omega}(\omega \oplus h)\\
	&=\sum_{k \leq 0}\sum_{h \in H^{\perp}}1_{B^{k}\Omega}(\omega \oplus h)
\end{align*}

\begin{thm}
A measurable set $S \subset G^{*}$ is a generalized scaling set if and only if the following conditions hold
\begin{enumerate}
\item[$(i)$] $\mu^*(S)=\frac{1}{p-1}$,
\item[$(ii)$] $S \subset BS$ (modulo null sets),
\item[$(iii)$] $lim_{k \rightarrow \infty}1_{S}(B^{-k}(\omega))=1$,  for a.e $\omega \in G^{*}$,
\item[$(iv)$] $\sum_{\sigma=0}^{p-1}\eta(\omega\oplus 0.\sigma)=\eta(B \omega)\oplus 1$ a.e, where $\eta(\omega)=\sum_{h \in H^{\perp}}1_{S}(\omega \oplus h)$.
\end{enumerate}
\end{thm}

\begin{proof}
We first assume that $S$ is a generalized scaling set. Condition (i) follows from the definition of generalized scaling set and (ii) follows from previous lemma as $S=\bigcup_{j=1}^{\infty}B^{-j}\Omega$, for some wavelet set $\Omega$. Further, the union $\bigcup_{j=1}^{\infty}B^{-j}\Omega$ is disjoint therefore
$$
1_{S}(B^{-k}\omega)=1_{\bigcup_{j=1}^{\infty}B^{-j}\Omega}(B^{-k}\omega)
=\sum_{j=-k+1}^{\infty}1_{\Omega}(B^{j}\omega).
$$

Using the fact that $\sum_{j \in \mathbb{Z}}1_{\Omega}(B^{j}\omega)=1$, for a.e $\omega$, we get $1_{S}(B^{-k}\omega)\rightarrow 1$, as $k \rightarrow \infty$ a.e. This proves (iii).
		
Note that $\eta(\omega)=\sum_{h \in H^{\perp}}1_{S}(\omega\oplus h)=\sum_{h \in H^{\perp}}\sum_{j=1}^{\infty}1_{\Omega}(B^{j}(\omega \oplus h))$.

\begin{align*}
\sum_{\sigma=0}^{p-1}\eta(\omega\oplus 0.\sigma)&=\sum_{h \in H^{\perp}}\sum_{j=1}^{\infty}\sum_{\sigma=0}^{p-1}1_{\Omega}(B^{j}(\omega\oplus h\oplus 0.\sigma))\\
		&=\sum_{h \in H^{\perp}}\sum_{j=0}^{\infty}\sum_{\sigma=0}^{p-1}1_{\Omega}(B^{j}(B \omega\oplus Bh \oplus \sigma.0))\\
		&=\sum_{h \in H^{\perp}}\sum_{j=0}^{\infty}1_{\Omega}(B^{j}(B\omega\oplus h))
\end{align*}

\begin{eqnarray}
		\sum_{\sigma=0}^{p-1}\eta(\omega\oplus 0.\sigma)=\sum_{h \in H^{\perp}}\sum_{j=1}^{\infty}1_{\Omega}(B^{j}(B\omega\oplus h))+\sum_{h \in H^{\perp}}1_{\Omega}(B\omega\oplus h)
\end{eqnarray}

\begin{align*}
		& \quad \quad =\eta(B\omega)\oplus 1, \text{ for a.e } \omega.
\end{align*}

Thus condition (iv) is also satisfied.
		
Next, suppose that conditions from (i) to (iv) are given.
		
Let $\Omega=BS \setminus S$, then from (ii) we get that $B^{k} \Omega \cap B^{j}\Omega=\phi$, for $j,k \in \mathbb{Z},j\neq k$ and this implies that $\lbrace B^{j} \Omega : j \in \mathbb Z\rbrace$ is a family of a.e mutually disjoint sets.
		
From (ii) we can write $\bigcup_{j=1}^{\infty}B^{-j}\Omega \subset S$ but both these sets have same measure so they are equal a.e. and hence
$$ 
1_{S}(\omega)=1_{\cup_{j=1}^{\infty}B^{-j}\Omega}(\omega)=\sum_{j=1}^{\infty}1_{\Omega}(B^{j}\omega), \text{ for a.e } \omega \in G^{*}
$$
that is, $1_{S}(B^{-k}\omega)=\sum_{j=-k+1}^{\infty}1_{\Omega}(B^{j}\omega)$ a.e. Then from (iii) we get, $\sum_{j \in \mathbb{Z}}1_{\Omega}(B^{j}\omega)=1 \text{ a.e }$.

From condition (iv) and equation (6), we have
$$ 
\sum_{h \in H^{\perp}}1_{\Omega}(B\omega\oplus h)=1 \text{ a.e. }
$$
Therefore $\Omega$ is a wavelet set and hence $S$ is a generalized scaling set.
\end{proof}

The following theorem  gives sufficient conditions for a set to be a generalized scaling set.

\begin{thm}
Suppose that a measurable set $S \subset G^{*}$ contains a neighbourhood of $0$, invariant under $B^{-1}$ and the corresponding characteristic function $1_{S}$ satisfies the equation 
\begin{equation*}
1+\sum_{h \in H^{\perp}}1_{S}(\omega \oplus h)=\sum_{h \in H^{\perp}} 1_{S}(B^{-1}(\omega \oplus h)) \text{ for  a.e.  } \omega,
\end{equation*}
then $\Omega=BS \backslash S$ is a wavelet set.
\end{thm}

\begin{proof}
From the above equation, we have
\begin{align*}
1 
		&=\sum_{h \in H^{\perp}} 1_{\Omega}(\omega \oplus h) 
\end{align*}
Thus translates of $\Omega$ cover $G^*$ disjointly a.e.
  Since $S$ is invariant under $B^{-1}$ and $B$ is one-one, Then the dilates of $\Omega$ under $B$ are disjoint. 
  Thus, $\Omega$ tiles $G^*$ a.e. under dilation. 	
\end{proof}



\begin{thebibliography}{}


\bibitem{behera} B. Behera and Q. Jahan, \lq\lq Characterization of wavelets and MRA wavelets on local fields of positive characteristics\rq\rq, Collect. Math. $\bf{66}$, 33--53(2013).	

\bibitem{bened1} J. J. Benedetto and R. L. Benedetto, \lq\lq A wavelet theory for local fields and related groups\rq\rq, J. Geom. Anal. $\bf{14}$, 423--456(2004).	

\bibitem{bened2} R. L. Benedetto, \lq\lq Examples of wavelets for local fields, Wavelets, Frames, and Operator Theory\rq\rq, Proc. Workshop, College Park, MD, 2003 (Am. Math. Soc., Providence, RI, 2004). Contemp. Math. $\bf{345}$, 27--47.

\bibitem{bownik} M. Bownik, Z. Rzeszotnik and D. Speegle, \lq\lq A characterization of dimension function of wavelets\rq\rq, Appl. Comput. Harmonic Anal. $\bf{10}$, 71-92(2001).

\bibitem{dai} X. Dai and D. R. Larson, \lq\lq Wandering vectors for unitary systems and orthogonal wavelets\rq\rq, Mem. Amer. Math. Soc. 134 no. 640, MR 98m: 47067(1998).

\bibitem{dai1} X. Dai,  D. R. Larson and D. M. Speegle, \lq\lq Wavelet sets in $\mathbb{R}^{n}$\rq\rq, J. Fourier Anal. Appl. $\bf{3}$, 451-456(1997).

\bibitem{farkov}  Y. A. Farkov, \lq\lq Multiresolution analysis and wavelets on Vilenkin Groups\rq\rq, Facta Univ. $\bf{21}$, 309--325(2008).

\bibitem{farkov1}Yu. A. Farkov and E. A. Rodionov, \lq\lq Algorithms for Wavelet Construction on Vilenkin Groups\rq\rq, p-Adic Numbers, Ultrametric Anal. and Appl. $\bf{3}$, 181–195(2011).

\bibitem{farkov2}Yu. A. Farkov, \lq\lq Orthogonal wavelets on direct products of cyclic groups\rq\rq, Math. Notes 82, 843–859 (2007).


\bibitem{fine} N. J. Fine, \lq\lq On the Walsh functions\rq\rq, Trans. Amer. Math. Soc. $\bf{65}$, 372-414(1949).

\bibitem{golu} B. I. Golubov, A. V. Efimov and V. A. Skvortsov, \textit{Walsh series and transforms: Theory and Applications} (Springer-Verlag, 2013).

\bibitem{hern} E. Hern\'andez and  G. Weiss, \textit{A First Course on wavelets} (CRC press, 1996).	

\bibitem{hern1}E. Hern\'andez,  X. Wang and  G. L. Weiss, \lq\lq Smoothing minimally supported frequency (MSF) wavelets: Part I\rq\rq, J. Fourier Anal. Appl. $\bf{2}$, 329--340(1996).

\bibitem{hern2} E. Hern\'andez,  X. Wang and  G. L. Weiss, \lq\lq Smoothing minimally supported frequency (MSF) wavelets: Part II\rq\rq, J. Fourier Anal. Appl. $\bf{3}$, 23--41 (1996).

\bibitem{Ionascu} E. J. Ionascu, D. Larson and C. Pearcy, \lq\lq On wavelet sets\rq\rq, J. Fourier Anal. Appl. $\bf{4}$, 711-721(1998).

\bibitem{lang1}W.C. Lang, \lq\lq Orthogonal wavelets on the Cantor dyadic group\rq\rq, SIAM J. Math. Anal. $\bf{27}$, 305-312(1996).

\bibitem{lang2} W.C. Lang, \lq\lq Wavelet analysis on the Cantor dyadic group\rq\rq,  Houston J. Math. $\bf{24}$, 533-544(1998).



\bibitem{rudin} W. Rudin, \textit{Fourier Analysis on Groups} (John Wiley \& Sons, New York,1962).

\end{thebibliography}
\end{document}